\documentclass[reqno,12pt]{amsart}
\usepackage[margin=1.1in, footskip=1cm]{geometry}
\usepackage{amsmath, amsthm, amssymb, booktabs}
\usepackage{mathrsfs}
\pdfpagewidth=\paperwidth
\pdfpageheight=\paperheight

\newtheorem{theorem}{Theorem}[section]
\newtheorem{lemma}[theorem]{Lemma}

\theoremstyle{definition}

\theoremstyle{remark}

\numberwithin{equation}{section}

\def\grm{{\mathfrak m}}

 \def\Del{{\Delta}}

\def\d{{\partial}}

\def\le{\leqslant} \def\ge{\geqslant}

\def\d{{\,{\rm d}}}

\makeatletter
\@namedef{subjclassname@2020}{\textup{2020} Mathematics Subject Classification}
\makeatother

\begin{document}
\title[Smooth Weyl sums]{Estimates for smooth Weyl sums on minor arcs}
\author[J\"org Br\"udern]{J\"org Br\"udern}
\address{Mathematisches Institut, Bunsenstrasse 3--5, D-37073 G\"ottingen, Germany}
\email{jbruede@gwdg.de}
\author[Trevor D. Wooley]{Trevor D. Wooley}
\address{Department of Mathematics, Purdue University, 150 N. University Street, West 
Lafayette, IN 47907-2067, USA}
\email{twooley@purdue.edu}
\thanks{The authors acknowledge support by Akademie der Wissenschaften zu G\"ottingen. First author supported by Deutsche Forschungsgemeinschaft 
Project Number 462335009. Second author supported by NSF grant DMS-2001549.}
\date{}

\begin{abstract} We provide new estimates for smooth Weyl sums on minor arcs and explore their consequences for the distribution of the  fractional 
parts of $\alpha n^k$. In particular, when $k\ge 6$ and $\rho(k)$ is defined via the relation $\rho(k)^{-1}=k(\log k+8.02113)$, then for all large numbers 
$N$ there is an integer $n$ with $1\le n\le N$ for which $\| \alpha n^k\|\le N^{-\rho(k)}$.\end{abstract}

\maketitle

\section{Introduction}
 Estimates for smooth Weyl sums on minor arcs play a prominent role in applications of the Hardy-Littlewood method, in the study of the distribution of 
fractional parts of polynomial sequences, and in many other branches of the theory of numbers. When $2\le R\le P$, let $\mathscr A(P,R)$ denote the set 
of natural numbers not exceeding $P$ having all of their prime factors bounded by $R$. Given a natural number $k\ge 2$, define the Weyl sum
\begin{equation}\label{1.1}
f(\alpha;P,R)=\sum_{n\in \mathscr A(P,R)}e(\alpha n^k),
\end{equation}
where, as usual, we write $e(z)$ for $\mathrm e^{2\pi {\rm i}z}$. In this context, a typical choice of minor arcs  is the set $\mathfrak n$ of all real numbers 
$\alpha$ with the property that when $q\in\mathbb N$ and $ a\in\mathbb Z$ are coprime with $|q\alpha-a|\le P^{1-k}$, one has $q>P$. Improving earlier 
estimates of Vaughan \cite{V89}, Wooley \cite{W95} showed that for each $\varepsilon>0$ there is a positive number $\eta=\eta(\varepsilon,k)$ such that 
uniformly in $2\le R\le P^\eta$ one has the estimate
\[
\sup_{\alpha\in\mathfrak n} |f(\alpha;P,R)|\ll P^{1-\rho(k)+\varepsilon},
\]
where $\rho(k)^{-1}=k(\log k + O(\log\log k))$. Our recent work \cite{BWCr} on Waring's problem yields progress on moment estimates for the Weyl sum 
$f(\alpha;P,R)$ that makes it possible to refine this bound, and make it more explicit.
 
\begin{theorem} \label{thm1.1}  Let $k\ge 6$, and define the positive number $\rho(k)$ by 
\begin{equation}\label{1.2}
\rho(k)^{-1} = k(\log k + 8.02113).
\end{equation}
Then, there is a positive number $\eta=\eta(k)$ with the property that uniformly in $2\le R\le P^\eta$ one has the estimate
\[
\sup_{\alpha\in\mathfrak n} |f(\alpha;P,R)|\ll P^{1-\rho(k)}.
\]
\end{theorem}

This new bound implies an improvement of \cite[Theorem 1.2]{W95} concerned with localised estimates for the fractional parts of $\alpha n^k$. 

\begin{theorem}\label{thm1.2} Suppose that $\alpha \in \mathbb R$. Let $k\ge 6$ and define $\rho(k)$ via \eqref{1.2}. Then, whenever $N$ is sufficiently 
large in terms of $k$, one has
\[
\min_{1\le n\le N} \|\alpha n^k\| \le N^{-\rho(k)}.
\]
\end{theorem}

Here, as is usual in this context, we write $\| \theta \|$ for $\min \{ |\theta -n|: n\in \mathbb Z\}$. For comparison, a similar conclusion is provided by 
\cite[Theorem 1.2]{W95} with $\rho(k)^{-1}=k(\log k+O(\log \log k))$. The numerical values for permissible $\rho(k)$ may be improved for small values of 
$k$. We direct the reader to Theorems \ref{thm4.1} and \ref{thm5.1} for explicit such refinements valid for $6\le k\le 20$. These conclusions are superior to 
all estimates hitherto available when $k\ge 10$.

Our proof of Theorem \ref{thm1.1} draws inspiration from the second author's earlier work \cite{W93, W95}, but also imports our more recent ideas 
through an estimate of Weyl's type that occurs as \cite[Theorem 3.5]{BWFr}. This bound is most powerful when the argument $\alpha$ is close to a fraction 
$a/q$ with $(a,q)=1$ and $q$ is of rough size $P^{k/2}$.  This results in  genuinely improved performance of the overall infrastructure underlying the 
proof of \cite[Theorem 1.1]{W95}. In addition, we improve the large sieve estimate embodied in \cite[Section 4]{W95}. The large sieve is replaced by a 
more direct use of the Sobolev-Gallagher inequality to remove an unwanted restriction to even moments of smooth Weyl sums in the final estimate 
(\cite[Lemma 4.1]{W95}). Having achieved the latter, we use the occasion to supply admissible exponents for moments of order $t$, with $t>4$ a real, not 
necessarily even number. This result, Theorem \ref{thm2.1} below, will prove useful in applications of major arcs moment estimates, as once again the 
restriction to even moments in optimisation procedures like those in \cite[Section 8]{PN1} or \cite[Section 6]{PN2} is certainly undesired, typically 
accommodated {\it a posteriori}, and now removable, at least for larger $k$.

\section{Admissible exponents}
Our goal in this section is to establish estimates for moments of smooth Weyl sums of sufficient flexibility that technical complications in our later 
applications may be avoided. For the remainder of the paper, we fix a natural number $k\ge 2$. Recall the definition \eqref{1.1}, and define the moment
\[
U_t(P,R) = \int_0^1 |f(\alpha;P,R)|^t\,\mathrm d \alpha,
\]
where $t$ is a non-negative real number. Following earlier convention, we say that the real number $\Delta_t$ is {\em admissible} (for $t$) when, for each 
$\varepsilon>0$, there exists $\eta>0$ having the property that, whenever $2\le R\le P^\eta$, one has
\[
U_t(P,R)\ll P^{t-k+\Delta_t+\varepsilon}. 
\]
We note that admissible exponents $\Delta_t$ are non-negative and may be chosen so that $\Delta_t\le k$.\par

In early work on moments of smooth Weyl sums admissible exponents are denoted differently. Since the discussion was focussed on moments of order 
$t=2s$ with $s\in\mathbb N$, the subscript of the exponent was often $s$, not $2s$, which would be in line with our definition. The reader should keep this 
in mind when comparing our findings with earlier results.\par

In order to simplify our exposition, we henceforth adopt the following convention concerning $\varepsilon$, $R$ and $\eta$. First, whenever $\varepsilon$ 
occurs in a statement, we assert that the statement holds for each positive value of $\varepsilon$. Implicit constants hidden in the symbols of Vinogradov 
and Landau may depend on the value assigned to $\varepsilon$. Second, should $R$ or $\eta$ appear in a statement, then it is asserted that the 
statement holds whenever $R\le P^\eta$ and $\eta$ is taken to be a positive number sufficiently small in terms of $\varepsilon$.\par

For each non-negative number $t$, we define the positive number $\delta_t$ to be the unique solution of the equation 
\begin{equation}\label{2.3}
\delta_t+ \log \delta_t = 1-t/k.
\end{equation}
The conclusion of \cite[Theorem 2.1]{W93} shows that when $k\ge 4$ and $t$ is an even integer with $t\ge 4$, then the exponent $\Delta_t=k\delta_t$ is 
admissible. Note that our earlier cautionary comment applies, as the quantity $\delta_{s,k}$ occuring in the statement of \cite[Theorem 2.1]{W93} is equal 
to our $\delta_{2s}$. Moreover, when $t=2s$ with $s$ a natural number, then it follows via orthogonality that $U_t(P,R)$ is equal to the number of solutions 
of the equation
\begin{equation}\label{2.4}
x_1^k+\cdots+ x_s^k = y_1^k+\cdots+ y_s^k,
\end{equation}
with $x_i,y_i\in \mathscr A(P,R)$, a quantity which in \cite{W93} is denoted $S_s(P,R)$. We now extend this earlier result to the situation in which $t$ is 
permitted to be any real number with $t\ge 4$.

\begin{theorem}\label{thm2.1}
Let $k\ge 6$ and suppose that $t$ is a real number with $t\ge 4$. Then the exponent $\Delta_t=k\delta_t$ is admissible.
\end{theorem}

This is a special case of a more general theorem in which a weighted analogue  of the exponential sum $f(\alpha;P,R)$ appears. When $w(n)\in\mathbb C$ 
$(n\in\mathbb N)$, we  define
\[
\|w\|_X = \max_{1\le n\le X} |w(n)|.
\]
Also, when $t$ is a positive number, put
\begin{equation}\label{2.6} 
U_t(P,R;w) = \int_0^1 \Big| \sum_{n\in \mathscr A(P,R)}w(n)e(\alpha n^k) \Big|^t\,\mathrm d \alpha .
\end{equation}
We say that the real number $\Delta_t$ is {\it weight-uniform admissible} when, for each $\varepsilon>0$, there exists $\eta>0$ having the property that, 
whenever $2\le R\le P^\eta$, uniformly in $w$ one has
\begin{equation}\label{2.1}
U_t(P,R;w) \ll\|w\|^t_P\, P^{t-k+\Delta_t+\varepsilon}.
\end{equation}
Note that when $s$ is a natural number,  it follows by orthogonality that
\[
U_{2s}(P,R;w) = \sum_{\mathbf x,\mathbf y} \prod_{i=1}^sw(x_i){\overline {w(y_i)}},
\]
where $x_i,y_i \in\mathscr A(P,R)$ $(1\le i\le s)$ are constrained by \eqref{2.4}. Thus, one has
\begin{equation}\label{2.7}
U_{2s} (P,R;w) \le \|w\|^{2s}_P \, U_{2s}(P,R),
\end{equation}
and hence the exponent $\Delta_{2s}$ is weight-uniform admissible whenever it is admissible.

\begin{theorem}
\label{thm2.2} Let $k\ge 6$ and suppose that $t$ is a real number with $t\ge 4$. Then the exponent $\Delta_t=k\delta_t$ is weight-uniform admissible.
\end{theorem}

\begin{proof}
Consider a natural number $s\ge 2$, and define the positive number $\Delta_{2s}$ to be the unique solution of the equation
\begin{equation}\label{2.8}
\frac{\Delta_{2s}}{k} + \log \frac{\Delta_{2s}}{k} = 1-\frac{2s}{k}-\frac{5}{16k^2}. 
\end{equation}
We assert that the  exponent $\Delta_{2s}$ is admissible. In order to confirm this assertion, observe first that when $s=2$ the exponent $k-2$ is admissible, 
as a consequence of Hua's lemma (see \cite[Lemma 2.5]{hlm}). Moreover, one has
\[
1-\frac{2}{k} + \log \Big(1-\frac2{k}\Big) < 1- \frac4{k} - \frac2{k^2} < 1 - \frac4{k} - \frac{5}{16k^2}.
\]
Hence, the exponent $\Delta_4$ defined via \eqref{2.8} is admissible. When $s$ is a natural number exceeding 2, meanwhile, it follows from the proof of 
\cite[Theorem 2.1]{W93} that an admissible exponent $\Delta$ for $2s$ exists satisfying
\[
\frac{\Delta}{k} + \log \frac{\Delta}{k} \le \delta_{2s-2} + \log \delta_{2s-2} - \frac2{k} + \frac{E}{2k^2},
\]
where $E\le k\,2^{2-k}-1\le -5/8$. This upper bound is a direct interpretation of the penultimate displayed equation of the proof of 
\cite[Theorem 2.1]{W93}, on page 167, with the inequality for $E$ immediately following the latter equation. In view of \eqref{2.3}, one has
\[
\frac{\Delta}{k} + \log \frac{\Delta}{k} \le 1 - \frac{2s}{k}- \frac{5}{16k^2}, 
\] 
and so a comparison with \eqref{2.8} confirms that $\Delta_{2s}$ is admissible.\par

Given a real number $t$ with $t\ge 4$, we put $s=\lfloor t/2\rfloor$ and $v= t/2 -s$, so that $t=2s+2v$. An application of H\"older's inequality leads from 
\eqref{2.6} via \eqref{2.7} to the bound
\begin{equation}\label{2.10}
U_t(P,R;w) \le \|w\|_P^t \, U_{2s}(P,R)^{1-v}\, U_{2s+2}(P,R)^v.
\end{equation}  
We have now effectively removed the weights from consideration.\par

We make use of the admissible exponent $\Delta_{2s}$  defined via \eqref{2.8}, but refine slightly the  admissible exponent $\Delta_{2s+2}$. Let 
$\omega= 2^{1-k}(1-\Delta_{2s}/k)$. The argument of the proof of \cite[Theorem 2.1]{W93} on page 167 shows that the positive number
\[
\Delta'_{2s+2} = \Delta_{2s} \Big(1- \frac{2-\omega}{k+\Delta_{2s}}\Big)
\]
is admissible for $2s+2$. In view of the upper bound \eqref{2.1}, we find from \eqref{2.10} that $\Delta_t$ is weight-uniform admissible, where
\begin{align*}
\Delta_t &= (1-v) \Delta_{2s} + v \Delta'_{2s+2} \\
&= (1-v) \Delta_{2s} + v \Delta_{2s} \Big(1- \frac{2-\omega}{k+\Delta_{2s}}\Big) \\
&= \Delta_{2s} \Big(1- v\,\frac{2-\omega}{k+\Delta_{2s}}\Big).
\end{align*}

Observe that
\[
\frac{\Delta_t}{k} + \log  \frac{\Delta_t}{k} =\frac{\Delta_{2s}}{k} + \log  \frac{\Delta_{2s} }{k} - v \, \frac{\Delta_{2s}}{k}\, 
\frac{2-\omega}{k+\Delta_{2s}}+ \log\Big( 1 - v \frac{2-\omega}{k+\Delta_{2s}}\Big) ,
\]
so that from \eqref{2.8} we obtain
\[ 
\frac{\Delta_t}{k} + \log  \frac{\Delta_t}{k} \le 1-\frac{2s}{k} - \frac5{16 k^2} - v\, \frac{\Delta_{2s}}{k}\, \frac{2-\omega}{k+\Delta_{2s}}- v\, 
 \frac{2-\omega}{k+\Delta_{2s}} - \frac{v^2(2-\omega)^2}{2(k+\Delta_{2s})^2}.
\]
Since 
\[
2-\omega = 2 - 2^{1-k}(1-\Delta_{2s}/k) \ge 2-(1-\Delta_{2s}/k) = 1+ \Delta_{2s}/k,
\]
we deduce that
\[
\frac{\Delta_t}{k} + \log  \frac{\Delta_t}{k} \le 1 - \frac{t}{k} + \frac{E}{2k^2},
\]
where
\[
E = -\frac58 +2kv\omega -v^2.
\]
Recalling that $k\ge 6$ and $0\le v\le 1$, we have
\[
E = -\frac58 +v(k\,2^{2-k} -v) <0,
\]
and so
\[
\frac{\Delta_t}{k} + \log  \frac{\Delta_t}{k} \le 1 - \frac{t}{k}.
\]
In view of relation \eqref{2.3}, it follows that the exponent $k\delta_t$ is weight-uniform admissible for $t$. This proves Theorem \ref{thm2.2}, and the 
conclusion of Theorem \ref{thm2.1} follows as a corollary.
\end{proof}

\section{A new upper bound for smooth Weyl sums}
We adapt the arguments of \cite{W95} so as to obtain a new minor arc estimate for the smooth Weyl sum $f(\alpha;P,R)$. We begin with an analogue of 
\cite[Lemma 4.1]{W95}. The strategy adopted in the proof of the latter makes use of the large sieve inequality to estimate an exponential sum stemming 
from an even power of $f(\alpha;P,R)$. Here, we adopt a more flexible approach, able to handle positive real powers of this exponential sum, by appealing 
directly to the Sobolev-Gallagher inequality.

\begin{lemma}\label{lem3.1}
Suppose that $1/2<\lambda<1$, and write $M=P^\lambda$. Let $\alpha\in \mathbb R$, and suppose that $a\in \mathbb Z$ 
and $q\in \mathbb N$ satisfy
\[
(a,q)=1,\quad |q\alpha -a|\le \tfrac{1}{2}(MR)^{-k},\quad q\le 2(MR)^k,
\]
and either $|q\alpha -a|>MP^{-k}$ or $q>MR$. Then, if $t$ is a real number with $t>k+1$ and $\Delta_t$ is weight-uniform admissible, one has
\[
f(\alpha;P,R)\ll M^{1+\varepsilon}+P^{1+\varepsilon}\bigl( M^{-1}(P/M)^{\Delta_t}(1+q(P/M)^{-k})\bigr)^{1/t}.
\]
\end{lemma}

\begin{proof} We begin our argument in the same manner as the proof of \cite[Lemma 4.1]{W95}. Define the set
\[
\mathscr B(M,\varpi,R)=\{ v\in \mathscr A(P,R)\cap (M,MR]: \text{$\varpi|v$, and $p|v$ implies $p\ge \varpi$}\},
\]
in which, here and henceforth, the letters $\varpi$ and $p$ both denote prime numbers. Then, from \cite[equation (4.1)]{W95}, we find that there exists an 
integer $d$ with $1\le d\le P/M$, a real number $\theta\in [0,1)$, and a prime number $\varpi$ with $\varpi\le R$, such that
\begin{equation}\label{3.1}
f(\alpha ;P,R)\ll M^{1+\varepsilon}+P^\varepsilon Rg(\alpha ;d,\varpi,\theta),
\end{equation}
where
\[
g(\alpha;d,\varpi,\theta)=\sum_{\substack{v\in \mathscr B(M/d,\varpi,R)\\ (v,q)=1}} \Biggl| \sum_{\substack{u\in \mathscr A(P/M,\varpi)\\ (u,q)=1}}
e(\alpha (uvd)^k+\theta u)\Biggr| .
\]

\par The argument of the proof of \cite[Lemma 4.1]{W95} shows that one can partition the integers $v\in \mathscr B(M/d,\varpi,R)$ with $(v,q)=1$ into $L$ 
classes $\mathscr V_1,\ldots ,\mathscr V_L$, with $L=O(q^\varepsilon d^k)$, having the following property. That is, for each $j$, as $v$ varies over 
$\mathscr V_j$, the real numbers $\alpha (vd)^k$ are spaced apart at least $\xi=\tfrac{1}{2}\min \{ q^{-1},(P/M)^{-k}\}$ modulo $1$. Define
\[
h(\alpha)=\sum_{\substack{u\in \mathscr A(P/M,\varpi)\\ (u,q)=1}}e(\alpha u^k+\theta u).
\]
Then an application of the Sobolev-Gallagher inequality (see \cite[Lemma 1.1]{Mon1971}) to the continuously differentiable function $|h(\beta)|^t$ reveals 
that for $1\le j\le L$, one has
\begin{equation}\label{3.2}
\sum_{v\in \mathscr V_j}|h(\alpha (vd)^k)|^t\ll \xi^{-1}\int_0^1|h(\beta)|^t\d\beta +\int_0^1|h(\beta)^{t-1}h'(\beta)|\d \beta .
\end{equation}

Define the weights
\[
w_1(n)=\begin{cases} e(\theta n),&\text{when $(n,q)=1$},\\
0,&\text{when $(n,q)>1$},\end{cases}
\]
and $w_2(n)=2\pi {\rm i}n^kw_1(n)$. Then it follows from \eqref{2.1} that, for each prime $\varpi\le R$, one has
\[
\int_0^1 |h(\beta)|^t\d\beta =U_t(P/M,\varpi;w_1)\ll (P/M)^{t-k+\Delta_t+\varepsilon}
\]
and
\[
\int_0^1 |h'(\beta)|^t\d\beta =U_t(P/M,\varpi;w_2)\ll (P/M)^{kt}(P/M)^{t-k+\Delta_t+\varepsilon}.
\]
An application of H\"older's inequality therefore leads from \eqref{3.2} to the bound
\begin{align*}
\sum_{v\in \mathscr V_j}|h(\alpha (vd)^k)|^t&\ll \xi^{-1}(P/M)^{t-k+\Delta_t+\varepsilon}+(P/M)^k(P/M)^{t-k+\Delta_t+\varepsilon}\\
&\ll \bigl( q+(P/M)^k\bigr) (P/M)^{t-k+\Delta_t+\varepsilon}.
\end{align*}
By summing these contributions from each set $\mathscr V_j$ for $1\le j\le L$, we thus obtain
\[
\sum_{\substack{ v\in \mathscr B(M/d,\varpi,R)\\ (v,q)=1}} |h(\alpha (vd)^k)|^t\ll q^\varepsilon d^k (P/M)^{t+\Delta_t+\varepsilon}(1+q(P/M)^{-k}).
\]
An application of H\"older's inequality to \eqref{3.1} therefore confirms that
\[
|f(\alpha;P,R)|^t\ll M^{t+\varepsilon}+P^\varepsilon d^k\bigl( 1+q(P/M)^{-k}\bigr) (M/d)^{t-1}(P/M)^{t+\Delta_t},
\]
whence, since $t>k+1$, we obtain the bound asserted in the statement of the lemma.
\end{proof}

Within the arguments to follow, we work with major arcs of various formats. Thus, when $1\le Q\le P^{k/2}$, let $\mathfrak M(Q)$ denote the union of the 
intervals
\[
\{\alpha\in[0,1) : |q\alpha -a| \le QP^{-k} \}
\]
with $0\le a\le q\le Q$ and $(q,a)=1$. We write $\mathfrak m(Q)=[0,1)\setminus \mathfrak M(Q)$ and $\mathfrak m= \mathfrak m(P)$. Note that 
$\mathfrak m = \mathfrak n \cap [0,1)$.\par

We next extract a minor arc estimate for smooth Weyl sums from Lemma \ref{lem3.1}. Given a family $(\Delta_s)_{s>0}$ of weight-uniform admissible 
exponents, we define the real number $\tau=\tau(k)$ by means of the relation
\[
\tau(k)=\max_{w\in \mathbb N}\frac{k-2\Delta_{2w}}{4w^2}.
\]
We have observed already that for $w\in \mathbb N$, any admissible exponent $\Del_{2w}$ is also weight-uniform admissible. Thus, as discussed in the 
preamble to \cite[Lemma 3.1]{BWFr}, one finds that $\tau(k)\le 1/(4k)$. This exponent is relevant to uniform estimates of Weyl type for $f(\alpha;P,R)$.

\begin{lemma}\label{lem3.2}
Suppose that $k\ge 2$. Then, uniformly in $1\le Q\le P^{k/2}$, one has the bound
\[
\sup_{\alpha \in \grm(Q)}|f(\alpha ;P,R)|\ll PQ^{\varepsilon -2\tau(k)/k}.
\]
In particular, writing $D=4.5139506$, one has
\[
\sup_{\alpha \in \grm(Q)}|f(\alpha ;P,R)|\ll PQ^{-1/(Dk^2)}.
\]
\end{lemma}

\begin{proof}
The respective conclusions are available from \cite[Lemma 3.3]{BWFr} and \cite[Lemma 3.4]{BWFr}.
\end{proof}

In order to state the next theorem, we introduce the real number $\sigma=\sigma(k)$, defined via the relation
\begin{equation}\label{3.3}
\sigma(k)^{-1}=\inf_{t>k+1}\biggl( t+\frac{1+\Delta_t}{2\tau(k)}\biggr) .
\end{equation}
We then define the associated quantity $\lambda=\lambda(k)$ by putting
\begin{equation}\label{3.4}
\lambda(k)=1-\frac{\sigma(k)}{2\tau(k)}.
\end{equation}

\begin{theorem}\label{thm3.3}
Suppose that $1/2<\lambda<1$. Then one has
\[
\sup_{\alpha\in \grm(P^\lambda R)}|f(\alpha ;P,R)|\ll P^{1-\sigma(k)+\varepsilon}.
\]
\end{theorem}

\begin{proof}
We put $M=P^\lambda$ and apply Lemma \ref{lem3.1}. By Dirichlet's approximation theorem, there exist $a\in \mathbb Z$ and $q\in \mathbb N$ with
\[
(a,q)=1,\quad q\le 2(MR)^k\quad \text{and}\quad |q\alpha -a|\le \tfrac{1}{2}(MR)^{-k}.
\]
When $\alpha \in \grm(P^\lambda R)$, it follows that either $q>MR$ or $|q\alpha -a|>MRP^{-k}$, and hence Lemma \ref{lem3.1} shows that for each 
$t>k+1$ and each weight-uniform admissible exponent $\Delta_t$, one has
\begin{align}
f(\alpha;P,R)&\ll P^{\lambda+\varepsilon}+P^{1+\varepsilon}\bigl( P^{-\lambda+(1-\lambda)\Delta_t}(1+qP^{-k(1-\lambda)})\bigr)^{1/t}\notag \\
&\ll P^{1-\sigma}+P^{1+\varepsilon}\bigl( P^{-1+\sigma (1+\Delta_t)/(2\tau)}(1+qP^{-k\sigma/(2\tau)})\bigr)^{1/t}.\label{3.5}
\end{align}
Observe that from \eqref{3.3} one has
\[
\sup_{t>k+1}\frac{1}{t}\biggl( \frac{1}{\sigma}-\frac{1+\Delta_t}{2\tau}\biggr) \ge 1,
\]
whence
\[
\sup_{t>k+1}\frac{1}{t}\biggl( 1-\frac{\sigma (1+\Delta_t)}{2\tau}\biggr) \ge \sigma .
\]
Hence, we deduce from \eqref{3.5} that when $q\le P^{k\sigma/(2\tau)}$, one has
\begin{equation}\label{3.6}
f(\alpha ;P,R)\ll P^{1-\sigma +\varepsilon}.
\end{equation}

\par It remains to handle the situation in which $q>P^{k\sigma/(2\tau)}$. Write $Q=\tfrac{1}{2}R^{-k}P^{k\sigma/(2\tau)}$, and note that one then has
\[
\tfrac{1}{2}(MR)^{-k}=\tfrac{1}{2}R^{-k}\bigl( P^{1-\sigma /(2\tau)}\bigr)^{-k}=QP^{-k}.
\]
Thus, we see that $\alpha \in \grm(Q)$, and hence Lemma \ref{lem3.2} delivers the bound
\begin{equation}\label{3.7}
f(\alpha ;P,R)\ll PQ^{\varepsilon -2\tau/k}\ll P^{1-\sigma +\varepsilon}R.
\end{equation}
In view of our conventions concerning $\varepsilon$ and $R$, the conclusion of the theorem follows on combining \eqref{3.6} and \eqref{3.7}.
\end{proof}

\section{The proof of Theorem 1.1}
The first goal of this section is to optimise parameters in Theorem \ref{thm3.3} so as to prove Theorem \ref{thm1.1}. 

\begin{proof}[The proof of Theorem 1.1] We assume throughout that $k\ge 6$. The second conclusion of Lemma \ref{lem3.2} shows that one can proceed 
using the value $\tau = 1/{(2Dk)}$. Also, as a consequence of Theorem \ref{thm2.2}, the exponent $\Delta_t=k\delta_t$ is weight-uniform admissible for 
$t\ge 4$, with $\delta_t$ defined by equation \eqref{2.3}. In particular, one has $\Delta_t\le k \mathrm e^{1-t/k}$. Thus, the exponent $\sigma$ defined in 
\eqref{3.3} satisfies the bound
\[
\sigma^{-1} \le \inf_{t>k+1} \big(t+Dk(1+k\mathrm e^{1-t/k})\big).
\]
One may verify that the infimum on the right hand side here is attained when $t=k\log k + k(1+\log D)$, and thus
\begin{align*}
\sigma^{-1} &\le k\log k + k(1+\log D) +Dk(1+1/D)\\
& = k\log k + k(D+2+\log D).
\end{align*}
We therefore conclude that one has
$
\sigma^{-1} \le k(\log k + \phi),
$
where $\phi = 8.0211233\ldots$. We have now proved that
\[
\sup_{\alpha\in\mathfrak m(P^\lambda R)} |f(\alpha;P,R)| \ll P^{1-\sigma+\varepsilon}.
\]
Here $\lambda$ is the number defined in \eqref{3.4}. Since $P^\lambda R \le P$ and $f$ has period $1$, this establishes a little more than is actually 
claimed in Theorem \ref{thm1.1}.
\end{proof}

By making use of the tables of admissible exponents  to be found in \cite{FIWP1} ($k=6$) and \cite{FIWP4} ($7\le k\le 20$), one may numerically compute the 
value of the exponent $\sigma(k)$ defined via \eqref{3.3}. In the table below, we record values of $2w$, and a corresponding admissible exponent 
$\Delta_{2w}$. These numbers are taken from \cite{FIWP1, FIWP4} where the numbers $\lambda_w=2w-k+\Delta_{2w}$ are tabulated. We also supply an 
upper bound for the number $T(k)$ having the property that $\tau(k)>T(k)^{-1}$. These data have been computed for $6\le k\le 13$ directly from the 
definition, and are tabulated for $14\le k\le 20$ in \cite{BWCr}. Here, we note a typographic error in the latter source. Thus, the heading $w$ in the second 
column of \cite[Table 2]{BWCr} should read $2w$ (in place of $w$). Finally, we report numbers $S(k)$ having the property that $\sigma(k)>S(k)^{-1}$. All 
figures are rounded up in the last digit displayed. We summarise these conclusions in the form of a theorem.

\begin{theorem}\label{thm4.1}
When $6\le k\le 20$, one has the bound
\[
\sup_{\alpha\in\mathfrak m} |f(\alpha;P,R)| \ll P^{1-\sigma(k)+\varepsilon}, 
\]
where $\sigma(k)>S(k)^{-1}$.
\end{theorem}

\begin{table}[h]
\begin{center}
\begin{tabular}{ccccccc}
\toprule
\ $k$ & $2w$  & $\Delta_{2w}$ & \ $T(k)$ & \ $t$ & $\Del_t$ & \ $S(k)$  \\
\toprule
\ $6$ & $10$ & $1.724697$ & \ $39.2064$ & \ $22$ & $0.086042$ & \ $43.2899$\\
\ $7$ & $12$ & $2.014382$ & \ $48.4647$ & \ $26$ & $0.192538$ & \ $54.8980$\\
\ $8$ & $14$ & $2.310600$ & \ $58.0088$ & \ $32$ & $0.189117$ & \ $66.4897$\\
\ $9$ & $16$ & $2.603928$ & \ $67.5080$ & \ $38$ & $0.190186$ & \ $78.1736$\\
$10$ & $18$ & $2.894572$ & \ $76.9440$ & \ $44$ & $0.192696$ & \ $89.8855$\\
$11$ & $20$ & $3.184973$ & \ $86.3921$ & \ $48$ & $0.241313$ & $101.6199$\\
$12$ & $22$ & $3.470081$ & \ $95.6553$ & \ $54$ & $0.239541$ & $113.2844$\\
$13$ & $24$ & $3.755717$ & $104.9455$ & \ $60$ & $0.239277$ & $125.0283$\\
$14$ & $26$ & $4.039939$ & $114.1869$ & \ $66$ & $0.240167$ & $136.8055$\\
$15$ & $28$ & $4.323087$ & $123.3903$ & \ $74$ & $0.209471$ & $148.6185$\\
$16$ & $30$ & $4.606286$ & $132.5981$ & \ $80$ & $0.213791$ & $160.4732$\\
$17$ & $32$ & $4.888677$ & $141.7763$ & \ $86$ & $0.218395$ & $172.3698$\\
$18$ & $34$ & $5.170691$ & $150.9411$ & \ $92$ & $0.223249$ & $184.3193$\\
$19$ & $36$ & $5.451758$ & $160.0695$ & \ $98$ & $0.228287$ & $196.3057$\\
$20$ & $38$ & $5.732224$ & $169.1748$ & $104$ & $0.233496$ & $208.3383$\\
\bottomrule
\end{tabular}\\[6pt]
\end{center}
\caption{Choice of parameters for $6\le k\le 20$.}\label{tab2}
\end{table}

When $k\ge 10$, the bounds supplied by Theorems \ref{thm1.1} and \ref{thm4.1} are superior to any previously available bound of Weyl's type for either 
smooth or classical Weyl sums. When $k\le 9$, however, the bound 
\[ 
\sup_{\alpha\in\mathfrak m} |f(\alpha;P,P)| \ll P^{1 - \frac{1}{k(k-1)}+\varepsilon},
\]
available via recent developments in Vinogradov's mean value theorem (see \cite{BDG2016, Woo2019}), provides superior exponents for classical Weyl 
sums. For smooth exponential sums, meanwhile, the estimates for $\sigma(k)$ in the table are still superior to those listed in \cite{FIWP4}. 

\section{The fractional part of $\alpha n^k$}
The proof of Theorem \ref{thm1.2} is achieved via a pedestrian modification of \cite[\S6]{W95}, by utilising Theorem \ref{thm3.3}. Let 
$\nu=(\sigma(k)-\rho(k))/2$, so that $0<\nu<\sigma(k)$. Then the only issue to check is that, with $H=P^{\sigma(k)-\nu}$, one has
\[
(HP^{\lambda(k)}R)^{k-1}P^{\lambda(k)-k}\ll P^{-\sigma(k)}.
\]
In view of \eqref{3.3} and \eqref{3.4}, however, one has
\[
(k-1)\sigma (k)+k\lambda(k)-k\le \biggl( k-1-\frac{k}{2\tau(k)}\biggr) \sigma (k)<-2\sigma (k).
\]
With plenty of room to spare, this suffices to confirm the validity of the argument corresponding to \cite[\S6]{W95}, and we find that
\[
\min_{1\le n\le N}\| \alpha n^k\|\ll N^{\nu-\sigma(k)}.
\]
Since $\sigma (k)>\rho(k)$, the desired conclusion follows.\par

By applying the same argument as described above for $10\le k\le 20$, using the explicit exponents calculated in the previous 
section as recorded in the table therein, one obtains the following conclusion.

\begin{theorem}\label{thm5.1}
Let $k$ be an integer with $10\le k\le 20$. Then, with the exponent $S(k)$ defined as in Table 1, one has
\[
\min_{1\le n\le N}\| \alpha n^k\| \ll N^{-1/S(k)}.
\]
\end{theorem}

This theorem improves on the earlier results of \cite{FIWP4} for $k\ge 10$. Such conclusions are also addressed by Baker in the 
discussion following the statement of \cite[Theorem 3]{Bak2016}. Our new conclusions recorded in Theorems \ref{thm1.2} and 
\ref{thm5.1} improve on the estimates recorded in part (ii) of the latter discussion for $k\ge 10$. In \cite[Theorem 2]{Bak2016}, Baker 
points out ({\it inter alia}) that the new conclusions available from recent breakthroughs on Vinogradov's mean value theorem 
(see \cite{BDG2016, Woo2019}, for example) yield the upper bound
\[
\min_{1\le n\le N}\| \alpha n^k\| \ll N^{\varepsilon -1/(k(k-1))}.
\]
On noting that the exponent $S(k)$ recorded in Table 1 exceeds $k(k-1)$ for $k\le 9$, we see that our new estimates are superior to 
those available via this progress on Vinogradov's mean value theorem only for $k\ge 10$.

\bibliographystyle{amsbracket}
\providecommand{\bysame}{\leavevmode\hbox to3em{\hrulefill}\thinspace}

\end{document}